\documentclass[12pt]{article}
\usepackage{amssymb, latexsym, amsmath}
\usepackage{amsthm}
\usepackage{enumerate}
\usepackage{graphicx}
\usepackage{hyperref}

\theoremstyle{definition}
\newtheorem{theorem}{Theorem}[section]
\newtheorem{lemma}[theorem]{Lemma}
\newtheorem{proposition}[theorem]{Proposition}
\newtheorem{definition}[theorem]{Definition}

\newtheorem{conjecture}[theorem]{Conjecture}

\allowdisplaybreaks

\newcommand{\prob}{\mathbb{P}}
\newcommand{\expect}{\mathbb{E}}

\begin{document}

\title{The adjacent vertex distinguishing total chromatic number}
\author{Tom Coker\thanks{
	Department of Pure Mathematics and Mathematical Statistics, 
	University of Cambridge, UK}
	 \and Karen Johannson\thanks{
			Department of Mathematical Sciences,
			University of Memphis, 
			Memphis, TN 38152, USA}}
\date{\today}

\maketitle

\begin{abstract}
A well-studied concept is that of the total chromatic number.  A proper total colouring of a graph is a colouring of both vertices and edges so that every pair of adjacent vertices receive different colours, every pair of adjacent edges receive different colours and every vertex and incident edge receive different colours.  This paper considers a strengthening  of this condition and examines the minimum number of colours required for a total colouring with the additional property that for any adjacent vertices $u$ and $v$, the sets of colours incident to $u$ is different from the set of colours incident to $v$.  It is shown that there is a constant $C$ so that for any graph $G$, there exists such a colouring using at most $\Delta(G) + C$ colours.
\end{abstract}

\section{Preliminaries}

\subsection{Definitions}

Let $G=(V,E)$ be a simple graph with no loops or multiple edges.  For $k \in \mathbb{Z}^+$, a map $\varphi: V \cup E \to \{1, 2, \ldots, k\} = [k]$ is called a \emph{proper total $k$-colouring of $G$} if{f} 
\begin{itemize}
	\item for every $u, v \in V$, if $uv \in E$, then $\varphi(u) \neq \varphi(v)$ and $\varphi(u) \neq \varphi(uv)$,
	\item and for every pair $uv, uw \in E$ of adjacent edges, $\varphi(uv) \neq \varphi(uw)$.
\end{itemize}
In other words, $\varphi|_{V}$ is a proper vertex colouring, $\varphi|_E$ is a proper edge colouring  and every vertex receives a colour different from the colour of each of its incident edges.

The \emph{total chromatic number of $G$}, denoted $\chi''(G)$, is the least $k$ for which there exists a proper total $k$-colouring of $G$.

The maximum degree of a graph $G$ is denoted, as usual,  by $\Delta(G)$.  Under any proper total colouring a vertex of maximum degree in $G$ receives a colour different from that of any of its edges and thus $\chi''(G) \geq \Delta(G) +1$.

\begin{definition}
Let $G=(V, E)$ be a graph and $\varphi$ be a proper total colouring of $G$.  For each $v \in V$ the \emph{colour set of $v$ (with respect to $\varphi$)} is
\[
C_{\varphi}(v) = \{\varphi(v)\} \cup \{\varphi(vw): w \in N(v)\}.
\]
A vertex $v \in V$ is said to be \emph{distinguished from $u$ by $\varphi$} if{f} $C_{\varphi}(u) \neq C_{\varphi}(v)$ and $\varphi$ is said to be \emph{adjacent vertex distinguishing} if{f} every pair of adjacent vertices in $G$ are distinguished from each other by $\varphi$.

The least $k$ for which $G$ has an adjacent vertex distinguishing total $k$-colouring is called the \emph{adjacent vertex distinguishing total chromatic number}, denoted $\chi_{at}(G)$.
\end{definition}

Suppose that $\varphi$ is a total colouring of $G$ with the property that any colour used for a vertex never appears on an edge.  Then $\varphi$ is adjacent vertex distinguishing since if a vertex $u$ receives colour $1$ and $v$ is a neighbour of $u$ then $1 \notin C_{\varphi}(v)$ since $1$ appears only on vertices and $\varphi(v) \neq 1$ since $\varphi$ is a proper colouring.  As usual, let $\chi(G)$ denote the chromatic number of $G$ and let $\chi'(G)$ denote the edge-chromatic number of $G$.  Then $\chi_{at}(G) \leq \chi(G) + \chi'(G)$ and thus, by Vizing's theorem and Brooks' theorem, if $G$ is not a complete graph or an odd cycle then $\chi_{at}(G) \leq 2\Delta(G)+1$.  The hope, however, is that by allowing any colour to appear on both edges and vertices, the number of necessary colours for a proper total adjacent vertex distinguishing colouring can be reduced from $\chi(G) + \chi'(G)$. 

\subsection{Vertex distinguishing edge colourings}

The study of proper colourings that induce different colour sets on different vertices was introduced independently by Aigner, Triesch and Tuza \cite{ATT90}; Burris and Schelp \cite{BS97};  and \v{C}ern\'{y}, Hor\v{n}\'{a}k and Sot\'{a}k \cite{CHS96}.  These three groups each examined the number of colours needed to properly edge colour a graph so that every vertex has a colour set different from that of every other vertex.

Zhang, Liu, and Wang \cite{LWZ02} relaxed this condition, examining proper edge colourings that distinguish pairs of adjacent vertices.

\begin{definition}
Given a graph $G=(V, E)$, the \emph{adjacent vertex distinguishing edge chromatic number}, denoted $\chi'_a(G)$ is the least $k$ such that there exists $\varphi$, a proper edge $k$-colouring of $G$, with the property that if $u, v \in V$ with $uv \in E$, then $\{\varphi(uw): w \in N(u)\} \neq \{\varphi(vz): z \in N(v)\}$.
\end{definition}

In their paper, Zhang \emph{et al.} determine the exact value of $\chi_a'(G)$ for several classes of graphs and conjecture that if $G$ is a connected graph with $V(G) \geq 6$, then $\chi_a'(G) \leq \Delta(G) +2$.

Balister, Gy\H{o}ri, Lehel and Schelp \cite{BGLS07} showed that if $G$ is a graph with $\Delta(G) = 3$ then $\chi'_a(G) \leq 5$.  They also showed that  for $G$, any bipartite graph, $\chi'_{a}(G) \leq \Delta +2$ and for $G$ any graph, $\chi'_a(G) \leq \Delta(G) + O(\log_2 \chi(G))$.  The upper bound on $\chi_a'(G)$ for arbitrary graphs was sharpened by Hatami \cite{hH05} who, using probabilistic techniques, showed that if $G$ is a graph with $\Delta(G) \geq 10^{20}$, then

\begin{equation}\label{E:avdcHatami}
\chi'_a(G) \leq \Delta(G) +300.
\end{equation}

\subsection{Total colourings}

The study of adjacent vertex distinguishing total colourings was first introduced by Zhang, Chen, Li, Yao, Lu and Wang \cite{ZC05} who determined precise values of $\chi_{at}$ for several classes of graphs, including cycles, complete graphs, complete bipartite graphs and trees, and made the following conjecture.

\begin{conjecture}\label{Conj:atupperbd}
For every graph $G$,
\[
\chi_{at}(G) \leq \Delta(G) + 3.
\]
\end{conjecture}

There are graphs that attain the upper bound in Conjecture \ref{Conj:atupperbd}.  For example, Zhang \emph{et al.} \cite{ZC05} showed that when $n$ is odd, $\chi_{at}(K_n) = n+2 = \Delta(K_n) +3$.

Since an adjacent vertex distinguishing total colouring is also a proper total colouring, for any graph $G$, $\chi''(G) \leq \chi_{at}(G)$.  While it has been conjectured, independently by both Behzad \cite{mB65} and Vizing \cite{vV68}, that $\chi''(G) \leq \Delta(G) + 2$, currently, the best-known upper bound for graphs with sufficiently large maximum degree was given by Molloy and Reed \cite{MR98} who showed that there exists a $\Delta_0$ such that if $G$ is any graph with $\Delta(G) \geq \Delta_0$, then 
\begin{equation}\label{E:totalchrom}
\chi''(G) \leq \Delta(G) + 10^{26}.
\end{equation}

While a proof of Conjecture \ref{Conj:atupperbd} would require a significant improvement on the known upper bound for the total chromatic number of an arbitrary graph, in the case $\Delta(G) = 3$, the conjecture has been verified, independently by Wang \cite{hW07}, Chen \cite{xC08} and Hulgan \cite{jH09}.  Hulgan, in fact, showed that for any graph $G$ with $\Delta(G)=3$, there is an adjacent vertex distinguishing total $6$-colouring of $G$ with the property that at most one colour appears on both edges and vertices.  For graphs of larger maximum degree, Liu, An, and Gao \cite{AGL09} showed that if $G$ is a graph with $\Delta(G) = \Delta$ sufficiently large and $\delta(G) \geq 32 \sqrt{\Delta \ln \Delta}$, then $\chi_{at}(G) \leq \Delta + 10^{26} + 2 \sqrt{\Delta \ln \Delta}$.  Further details on the history of the problem can be found, for example, in Hulgan \cite{jH10}.

\subsection{Results}

Following an argument similar to that used by Hatami \cite{hH05} to prove the upper bound given in  equation (\ref{E:avdcHatami}), in this paper, a proof is given for the following upper bound on the $\chi_{at}(G)$.

\begin{theorem}\label{T:mainthm}
There exists $C_0 >0$ such that for every graph $G$, 
\[
\chi_{at}(G) \leq \chi''(G) + C_0.
\]
\end{theorem}

Applying Molloy and Reed's \cite{MR98} upper bound on $\chi''(G)$, yields an upper bound on $\chi_{at}(G)$ in terms of $\Delta(G)$.

\begin{theorem}
There exists $C'>0$ such that for every graph $G$, 
\[
\chi_{at}(G) \leq \Delta(G) + C'.
\]
\end{theorem}

The idea of the proof of Theorem \ref{T:mainthm} is to begin with a proper total colouring and recolour of some of the vertices and edges so that the resulting colouring remains a proper total colouring and becomes adjacent vertex distinguishing, but in such a way that only a constant number of new colours are added.  While the process of recolouring vertices is deterministic, the edges to be recoloured are chosen at random and probabilistic techniques are used to show that there is a `good' choice of edges for recolouring in a way to obtain an adjacent vertex distinguishing total colouring.

In Section \ref{S:probtools}, a few standard probabilistic results that are each used repeatedly are stated.  In Section \ref{S:lowdeg} it is proved that any proper total colouring can be redefined on the vertices to obtain a proper total colouring that distinguishes vertices of degree at most $\Delta(G)/2$ from their neighbours.  In Section \ref{S:highdeg}, it is shown that given any proper total colouring, there is a subset of the edges that can be recoloured with no more than a constant number of new colours so that vertices of degree at least $\Delta(G)/2+1$ are distinguished from their neighbours.  Finally, in Section \ref{S:mainthmpf}, these previous two results are combined to prove Theorem \ref{T:mainthm}.

The proof that the edges of $G$ can be recoloured appropriately requires an assumption that the maximum degree of $G$ is at least as large as a fixed constant.  However, once Theorem \ref{T:mainthm} is proved for graphs with sufficiently large maximum degree, it immediately holds true for all graphs, potentially with a larger constant $C_0$.

Since different techniques are applied to the subgraph induced by the vertices  of `low degree' and to that induced by the vertices of `high degree', it will be convenient to use the following notation.

\begin{definition}\label{D:lowdeg-highdeg}
For any graph $G=(V,E)$, set 
\begin{align*}
V_{\ell} 	&= \{v \in V: \deg(v) \leq \Delta(G)/2\} \text{ and}\\
V_h 			&= \{v \in V: \deg(v) > \Delta(G)/2\}.
\end{align*}
\end{definition}

Throughout, the following notation for graphs is used.  For any graph $G=(V,E)$ and for sets $A, B \subseteq V$, not necessarily disjoint, let the set of edges between $A$ and $B$ be
$E(A, B) = \{uv \in E: u \in A \text{ and } v \in B\}$.

  Given $F \subseteq E$, let $G[F]$ be the subgraph of $G$ induced by the edges in $F$.  For $v \in V$, let the degree of $v$ in $F$ be $\deg_{F}(v)=|\{f \in F: v \in f\}|$ and denote by $F(v) = \{vw \in F: w \in N(v)\}$, the edges of $F$ that are incident to $v$.  

If $\varphi$ is a total colouring of $G$ and $D \subseteq V \cup E$, then let the set of colours appearing in $D$ be $\varphi[D] = \{\varphi(v): v \in D \cap V\} \cup \{\varphi(uv): uv \in E \cap D\}$.

\section{Probabilistic tools}\label{S:probtools}

The following lemma gives estimates for the unlikelihood of a binomial random variable being either much larger or much smaller than its mean.  In the following form, it can be found, for example, in \cite[pp 267--268]{AS00}.
\begin{lemma}\label{L:binomtail}
Let $X$ be a binomial random variable with parameters $n\in \mathbb{Z}^+$ and $p \in (0,1)$.  For $pn < m <n$,
\[
\prob(X \geq m) 	\leq e^{m-np} \left(\frac{np}{m}\right)^m
\]
 and for $0< m < pn$,
\[
\prob(X  < m) 	\leq e^{-(m-np)^2/2pn}.
\]
\end{lemma}

The next theorem, due to Erd\H{o}s and Lov\'{a}sz \cite{EL75}, is known as the \emph{Lov\'{a}sz Local Lemma}.  It provides estimates on the probability of many events occurring simultaneously in a probability space.  For the form below, see, for example, \cite[pp 64--65]{AS00}.

\begin{theorem}\label{T:LLL}
Let $A_1, \ldots, A_n$ be events in a probability space and for each $i \in [n]$, let $\Gamma(i) \subseteq [n]$ be such that $A_i$ is mutually independent of the events $\{A_j: j \in [n] \setminus (\Gamma(i) \cup \{i\})\}$.
If there are $x_1, \ldots, x_n \in [0,1)$ such that for all $i \in
[1,n]$,
\[
\prob(A_i) < x_i \prod_{j \in \Gamma(i)}{(1-x_j)}
\]
then
\[
\prob\left(\bigcap_{i=1}^n{\bar{A}_i}\right) \geq \prod_{i=1}^n{(1-x_i)}
\]
and in particular, $\prob(\cap_{i=1}^n{\bar{A}_i})>0$.

As a special case, if there are $p \in (0,1)$ and $d \in \mathbb{Z}^+$ with the property that for each $i \in [n]$, $\prob(A_i) \leq p$, $|\Gamma(i)| \leq d$, and $p(d+1)e \leq 1$, then $\prob(\cap_{i=1}^n{\bar{A}_i})>0$.
\end{theorem}

The following inequality, due to McDiarmid and Reed \cite{McR06}, is a variation of a concentration result by Talagrand \cite{mT95}.  For further details on the Talagrand inequality see, for example, Talagrand \cite{mT96}.  

\begin{theorem}\label{T:tal}
Fix $c >0$, $r\geq 0$ and $d\geq 0$. Suppose that $g$ is a non-negative random variable with mean $\mu$ and $g=g(X_1, \ldots, X_n)$ where $X_1, \ldots, X_n$ are independent Bernoulli 0-1 random variables and
\begin{enumerate}[(a)]
	\item if $\mathbf{x}, \mathbf{x}' \in \{0,1\}^n$ differ in exactly one coordinate, then $|g(\mathbf{x})-g(\mathbf{x}')| \leq c$ and
	\item for any $s \geq 0$, if $g(\mathbf{y}) \geq s$, there is a set $I \subseteq [1,n]$ with $|I| \leq rs +d$ such that if $\mathbf{y}' \in \{0,1\}^n$ agrees with $\mathbf{y}$ on the coordinates in $I$, then $g(\mathbf{y}') \geq s$.
\end{enumerate}
Then, for any $t \geq 0$
\begin{align*}
\prob(g - \mu \geq t) 	& \leq e^{-\frac{t^2}{2c^2(r\mu + d + rt)}}\\
\prob(g - \mu \leq -t) 	& \leq e^{-\frac{t^2}{2c^2(r\mu + d + t/3c)}}.
\end{align*}
\end{theorem}

\section{Vertices of low degree}\label{S:lowdeg}

Since the vertices of low degree, as in Definition \ref{D:lowdeg-highdeg}, in a graph $G$ have relatively few neighbours compared to $\Delta(G)$, any total colouring of $G$ with more than $\Delta(G)$ colours can be adjusted by recolouring some vertices so that every vertex of $V_{\ell}$ is distinguished from all of its neighbours.  Recall that since $\chi''(G) \geq \Delta(G) +1$, if $\varphi$ is a proper total $k$-colouring of $G$, then $k \geq \Delta(G) + 1$.

\begin{proposition}\label{P:lowdeg}
Let $G = (V,E)$ be a graph and let $\varphi$ be a proper total colouring of $G$.  There exists a proper total colouring $\varphi'$ of $G$ with $\varphi'|_{E \cup V_h} = \varphi|_{E \cup V_h}$ such that for every $v \in V_{\ell}$, the colouring $\varphi'$ distinguishes $v$ from each of its neighbours.
\end{proposition}

\begin{proof}
Fix a graph $G$ and $\varphi$, a proper total $k$-colouring of $G$.  Let $\psi_0$ be a proper total $k$-colouring of $G$ with the property that among the proper total $k$-colourings of $G$ that agree with $\varphi$ on $E \cup V_h$, the map $\psi_0$ has the fewest vertices in $V_{\ell}$ not distinguished from one of its neighbours.  More precisely, among the total colourings
\[
\{\psi: \psi \text{ is a proper total $k$-colouring of $G$ with } \psi|_{E\cup V_h} = \varphi|_{E \cup V_h}\}
\]
$\psi_0$ is such that the quantity 
\[
|\{u \in V_{\ell}: \exists\ v \in N(u) \cap V_{\ell} \text{ with } C_{\psi_0}(u) = C_{\psi_0}(v)\}|
\]
is minimised.  It will be shown that, in fact, every vertex in $V_{\ell}$ is distinguished from all of its neighbours with respect to $\psi_0$.  Note that every vertex $v \in V_{\ell}$ is distinguished from every $u \in N(v) \cap V_h$ since $|C_{\psi_0}(v)|<|C_{\psi_0}(u)|$.

Suppose that there is a $u \in V_{\ell}$ not distinguished by $\psi_0$ from one of its neighbours.  The vertex $u$ will be recoloured so that the resulting total colouring is both proper and distinguishes $u$ from all of its neighbours.

If $i \in [k]$ is such that there is $v \in N(u)$ with $C_{\psi_0}(v) = \{i\} \cup C_{\psi_0}(u) \setminus \{\psi_0(u)\}$ and $i \neq \psi_0(v)$ then $\psi_0(v) \in \{\psi_0(uw): w \in N(u)\}$.  So, for every $v \in N(u)$, there is at most one colour $i_v \in [k] \setminus \{\psi_0(uw): w \in N(u)\}$ such that either $i_v = \psi_0(v)$ or $C_{\psi_0}(v) = \{i_v\} \cup C_{\psi_0}(u) \setminus \{\psi_0(u)\}$.

\begin{figure}[h]
	\centerline{\includegraphics{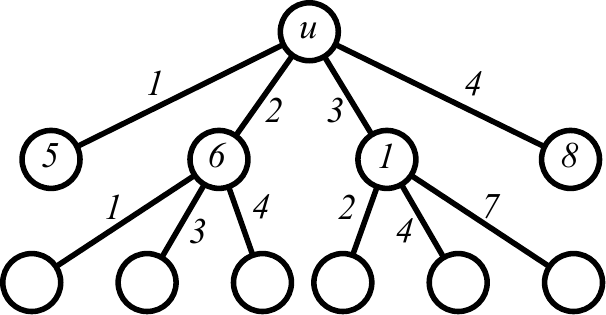}}
	\caption{Example of colours unavailable for vertex $u$}\label{F:lowdeg}
\end{figure}

Thus
\begin{align*}
\Big\vert \bigcup_{w \in N(u)}&\{\psi_0(w), \psi_0(uw)\}\\
	 & \hspace{-0.25in}\bigcup \{i \in [k]: \exists\ v 		\in N(u) \text{ with } C_{\psi_0}(v) = \{i\}\cup C_{\psi_0}(u) 			\setminus \{\psi_0(u)\}\} \Big\vert\\
		&\leq 2 \deg(u) \leq 2 \Delta/2 < k.
\end{align*}
Therefore, there is at least one colour $i_u \in [k]$ such that if $v \in N(u)$, then $i_u \neq \psi_0(v)$, $i_u \neq \psi_0(uv)$ and $C_{\psi_0}(v) \neq \{i_u\} \cup \{\psi_0(uw): w \in N(u)\}$.

Define $\psi_1: V\cup E \to [k]$ for each $x \in V\cup E$ by 
\[
\psi_1(x) = 
	\begin{cases}
		i_u 			&\text{if $x = u$},\\
		\psi_0(x)	&\text{otherwise}.
	\end{cases}
\]
Then, $\psi_1$ is a proper total colouring of $G$ with $k$ colours, $\psi_1|_{E \cup V_h} = \psi_0|_{E \cup V_h} = \varphi|_{E \cup V_h}$ and $\psi_1$ has fewer vertices in $V_{\ell}$ that are not distinguished from one of its neighbours than $\psi_0$ does.

Thus, $\psi_0$ is a proper total $k$-colouring of $G$ with $\psi_0|_{E \cup V_h} = \varphi|_{E \cup V_h}$ and for every $u \in V_{\ell}$ and $v \in N(u)$, $C_{\psi_0}(u) \neq C_{\psi_0}(v)$.
\end{proof}

\section{Vertices of high degree}\label{S:highdeg}

\begin{proposition}\label{P:highdeg}
There exists $\Delta_0 > 0$ and $C_1 > 0$ such that for every graph $G$ with $\Delta(G) \geq \Delta_0$ and $\varphi$, a proper total $k$-colouring of $G$, there is a proper total $(k+C_1)$-colouring, $\varphi'$, of $G$ such that for every $u,v \in V_h$, if $uv \in E$, then $C_{\varphi'}(u) \neq C_{\varphi'}(v)$.
\end{proposition}

This is proved in two steps.  First it is shown that there is a set of edges in $G$ that can be deleted so that, in the resulting subgraph, if a vertex is not incident to too few deleted edges, it is distinguished from its neighbours and so that every vertex has relatively few neighbours that were potentially not distinguished from some neighbour.

\begin{lemma}\label{L:highdegA}
Fix $m, d\in \mathbb{Z}^+$ with $m \geq d+4$, and $\varepsilon > 0$.  There exists $M > 0$ and $\Delta_2 > 0$ such that for every graph $G$ with $\Delta(G) \geq \Delta_2$ and $\varphi$, a proper total $k$-colouring of $G$, there is a set $E_1 \subseteq E(V_h, V)$ such that for each $v \in V$, $\deg_{E_1}(v) \leq M$, and setting $\varphi_1 = \varphi|_{V \cup E \setminus E_1}$,
\begin{enumerate}[(a)]
	\item for $u, v \in V_h$ with $uv \in E$ and $\deg(u) = \deg(v)$, if $\deg_{E_1}(u) \geq m$, then $|C_{\varphi_1}(u) \triangle C_{\varphi_1}(v)| \geq d$ and
	\item if $v \in V_h$, then $|\{u \in N_G(v): \deg_{E_1}(u) < m\}| \leq \varepsilon \Delta(G)$.
\end{enumerate}
\end{lemma}

\begin{proof}
Let $G$ be a graph and set $\Delta(G) = \Delta$.  Set $\lambda = 2(1+ \sqrt{2})(m + \ln(3/\varepsilon))$ and $M = 2 e \lambda$.

Set $p = \lambda/\Delta$ and select $X \subseteq E(V_h,V)$ randomly, with each edge in $E(V_h, V)$ included in $X$ independently with probability $p$.

Set $E_1 = E_1(X) = X \setminus \{uv \in E: \deg_{X}(u) > M\}$ so that every vertex is contained in at most $M$ edges from $E_1$.  

\begin{figure}[h]
	\centerline{\includegraphics{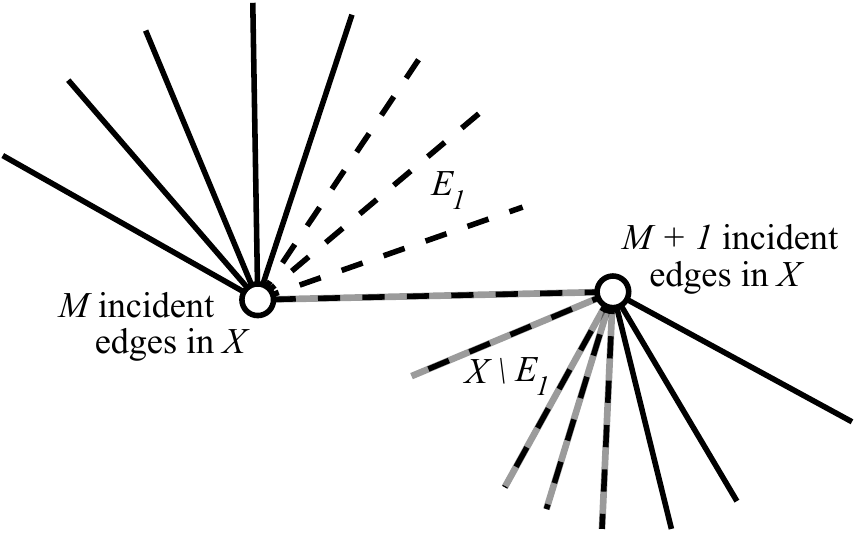}}
	\caption{Edges contained in $X$ and $E_1$.}
	\label{F:deletededges}
\end{figure}

For every $v \in V_h$ and $u \in N(v) \cap V_h$, define the following events, depending on the randomly chosen set of edges $X$:
\begin{align*}
A_{uv}	&= \left\{\deg_{E_1}(u) \geq m \text{ and } |C_{\varphi_1}(u) \triangle C_{\varphi_1}(v)| < d\right\}\\
B_v 		&= \left\{|\{u \in N_G(v): \deg_{E_1}(u) > m\}| > \Delta \varepsilon\right\}.
\end{align*}

The Lov\'{a}sz Local Lemma (Theorem \ref{T:LLL}) is used to show that
\[
\prob\left(\bigcap_{uv \in E(V_h, V_h)}\overline{A_{uv}} \cap \bigcap_{v \in V_h}\overline{B_{v}}\right) > 0.
\]
That is, with positive probability, the set $E_1$ satisfies both conditions $(a)$ and $(b)$.  For this, estimates on $\prob(A_{uv})$ and $\prob(B_v)$ are required.\\

\noindent \emph{Claim:} For each $v \in V_h$ and $u \in N(v) \cap V_h$ if $\deg(v) = \deg(u)$, then $\prob(A_{uv}) \leq 2^{2M + d} p^{m-d+1}$.\\

\noindent \emph{Proof of Claim:} Fix $u,v \in V_h$ with $uv \in E$ and $\deg(u) = \deg(v)$.  In order to estimate $\prob(A_{uv})$, it is convenient to condition on the following event.  For any $D \subseteq C_{\varphi}(u) \setminus \{\varphi(u)\}$ with $m \leq |D| \leq M$, let $Z_D$ be the event that $\varphi[\{uw \in E_1: w \in N(u)\}] = D$.  That is, $Z_D$ is the event that $D$ is the set of colours contributed to $C_{\varphi}(u)$ by $E_1$.  Given a set $D$, set $t = t(D) = |D \setminus C_{\varphi}(v)|$ and $\ell = \ell(D) = |D \cap C_{\varphi}(v)|$.

Fix such a set $D$, let $t = t(D)$ and $\ell = \ell(D)$ and suppose that $Z_D$ holds. Set $s = |C_{\varphi}(u) \setminus C_{\varphi}(v)|$.  Since $\deg(u) = \deg(v)$, then $|C_{\varphi}(v) \setminus C_{\varphi}(u)| = s$ also.

Since $s-t \leq |C_{\varphi_1}(u) \setminus C_{\varphi}(v)|$, if $s-t \geq d$, then
\[
\prob(|C_{\varphi_1}(u) \triangle C_{\varphi_1}(v)| < d\ |\ Z_{D}) = 0.
\]

  From now on, assume that $s-t < d$ and so $s < t+ d \leq t+ \ell + d \leq M+d$.  Supposing that $|C_{\varphi_1}(u) \triangle C_{\varphi_1}(v)| \leq d -1 < d$ holds, 
\begin{align*}
d -1	&\geq |C_{\varphi_1}(u) \triangle C_{\varphi_1}(v)|\\
	&= |C_{\varphi_1}(u) \setminus C_{\varphi}(v)| + |C_{\varphi_1}(u) \cap (C_{\varphi}(v) \setminus C_{\varphi_1}(v))| +\\
	& \hspace{1.5in} + |C_{\varphi_1}(v) \setminus C_{\varphi}(u)|+ |D \cap C_{\varphi_1}(v)|\\
	&\geq (s-t) + (s-|C_{\varphi}(v) \setminus (C_{\varphi_1}(v) \cup C_{\varphi}(u))|) +\\
	&\hspace{1.5in} + (\ell - |(C_{\varphi}(v) \setminus C_{\varphi_1}(v)) \cap D|). 
\end{align*}
Therefore,
\begin{align*}
|C_{\varphi}(v) \setminus (C_{\varphi_1}(v) \cup C_{\varphi}(u))| &+
	|(C_{\varphi}(v) \setminus C_{\varphi_1}(v))\cap D|\\
	&=|\varphi[E_1(v)] \setminus C_{\varphi}(u)| + |\varphi[E_1(v)] \cap D|\\
	&\geq (s-t) + s + \ell -d + 1\\
	&= 2(s-t) + t + \ell -d + 1\\
	&\geq m-d + 1.
\end{align*}

That is, of the colours deleted from $C_{\varphi}(v)$ when the edges in $E_1$ are removed, at least $m-d + 1$ are either from $C_{\varphi}(v) \setminus C_{\varphi}(u)$ or else $D \cap C_{\varphi}(v)$.  Since $|C_{\varphi}(v) \setminus C_{\varphi}(u)| = s$ and $|D \cap C_{\varphi}(v)| = \ell$, 
\begin{align*}
\prob(|C_{\varphi_1}(u) \triangle & C_{\varphi_1}(v)| < d|\, Z_{D})\\
		&\leq \prob(|\varphi[E_1(v)] \setminus C_{\varphi}(u)| + |\varphi[E_1(v)] \cap D| \geq m-d+1 |\, Z_{D})\\
		&\leq 2^{\ell} 2^s  p^{m-d+1}\\
		&\leq  2^M 2^{M+d} p^{m-d+1}\\
		&=2^{2M+d} p^{m-d+1}\\
\end{align*}
uniformly, for all choices of $D$.  Thus, for each $uv \in E(V_h)$,
\[
\prob(A_{uv}) \leq 2^{2M+d} p^{m-d+1}.
\]

\noindent \emph{Claim:}  There exists a constant $c_0$ such that if $v \in V_h$ then $\prob(B_v) \leq 3e^{-c_0 \Delta}$.\\

\noindent \emph{Proof of Claim:}  Given $v \in V_h$, consider the event $B_v$ that $|\{u \in N(v): \deg_{E_1} < m\}| \leq \varepsilon \Delta$.  Set
\begin{align*}
V_m &= \{u \in V_h:\ \deg_X(u) <m\}\\
V_M &= \{u \in V_h: \deg_X(u) > M\}\\
V_N &= \{u \in V_h: m \leq \deg_X(u) \leq M \text{ and } \exists\ w \in V_M \cap N(u) \text{ with } uw \in X \}.
\end{align*}
Then, $\{u \in V_h: \deg_{E_1}(u) < m\} \subseteq V_m \cup V_M \cup V_N$ and so by the pigeonhole principle, if $|\{u \in N(v): \deg_{E_1}(u) < m\}| > \varepsilon \Delta$, then either $|N(v) \cap V_m| > \varepsilon \Delta/3$, $|N(v) \cap V_M| > \varepsilon \Delta/3$, or else $|N(v) \cap V_N| \geq \varepsilon \Delta/3$. The probability of each of these three events occurring are considered separately, although the calculations are all similar.\\

\noindent \textbf{Case 1}: Consider the event $|N(v) \cap V_M| > \Delta \varepsilon/3$.  Note that for each $w \in V_M$, the quantity $\deg_X(w)$ is a binomial random variable with parameters $\deg_G(w)$ and $p$. 
\begin{align*}
\expect(|N(v) \cap V_M|)	
	&=\sum_{w \in N(v)} \prob(w \in V_M)\\
	&=\sum_{w \in N(v)} \prob(\deg_X(w) > M)\\
	&\leq \sum_{w \in N(v)} e^{M-p\deg_G(w)}\left(\frac{p\deg_G(w)}{M}\right)^M  &&\text{(by Lemma \ref{L:binomtail})}\\
	&\leq \Delta e^{M-\lambda/2} \left(\frac{\lambda}{M}\right)^M.
			&&\hspace{-0.7in} \text{(since $\lambda/2 \leq p\deg_G(w) \leq \lambda$)}	
\end{align*}

Changing the status of any one edge in $X$ changes the size of the set $V_M$ by at most $2$ and if $|N(v) \cap V_M| \geq a$, this event can be certified by the status of a collection of at most $M a$ edges.  Since $M = 2e \lambda$ and $\lambda \geq 2 \ln(3/\varepsilon)$, it follows that $e^{M - \lambda/2}(\lambda/M)^M \leq e^{-\lambda/2} 1/2^M < \varepsilon/3$.  Thus, by Theorem \ref{T:tal},
\begin{align*}
\prob(|N(v) \cap V_M| \geq \Delta \varepsilon/3)
	&\leq \exp\left({-\frac{(\Delta \varepsilon/3 - \Delta e^{-\lambda/2}(e\lambda/M)^M)^2}{2\cdot 2^2 M \Delta \varepsilon/3}}\right)\\
	&=\exp\left({-\Delta 3\frac{(\varepsilon/3 - e^{-\lambda/2}(e\lambda/M)^M)^2}{8M\varepsilon}}\right).
\end{align*}\\

\noindent \textbf{Case 2}: Now, consider the event $|N(v) \cap V_N| \geq \varepsilon \Delta/3$.  Fix $u \in N(v)$.  Then,
\begin{align*}
\prob(u \notin (V_M \cup V_m) &\text{ and } \exists w \in N(u) \cap V_M \text{ with } uw \in X)\\
	&\leq \prob(\exists\ w \in N(u) \cap V_M \text{ with } uw \in X)\\
	&\leq \sum_{w \in N(u)} \prob(w \in N_M \text{ and } uw \in X)\\
	&=\sum_{w \in N(u)} \prob(w \in N_M |\ uw \in X) \prob(uw \in X)\\
	&=\sum_{w \in N(u)} e^{M-\lambda/2} \left(\frac{\lambda}{M}\right)^M p	&&\text{(as above)}\\
	&\leq \Delta p e^{M-\lambda/2} \left(\frac{\lambda}{M}\right)^M\\
	&= \lambda e^{-\lambda/2}\left(\frac{\lambda e}{M} \right)^M.
\end{align*}
Thus,
\[
\expect(|\{u \in N(v) \setminus (V_M \cup V_m): \exists\ uw \in X \text{ with}\ w \in V_M\}|) \leq \Delta \lambda e^{-\lambda/2}\left(\frac{\lambda e}{M} \right)^M.
\]

Changing the status of one edge with respect to $X$  changes the value of $|N(v) \cap V_N|$ by at most $2M$ (since the only cases where anything changes are when some vertex is adjacent to exactly $M$ or $M+1$ edges in $X$).  As before, the event that  $|N(v) \cap V_N|\geq a$ can be certified by the status of a collection of at most $Ma$ edges.  By the choice of $M$ and $\lambda$, $\varepsilon/3 > \lambda e^{-\lambda/2}\left(\frac{\lambda e}{M} \right)^M$ and so by Theorem \ref{T:tal},
\begin{align*}
\prob(|N(v) \cap V_N| > \Delta \varepsilon/3)
	&\leq \exp\left({-\frac{(\Delta \varepsilon/3 - \Delta \lambda e^{-\lambda/2}\left(\frac{\lambda e}{M} \right)^M)^2}{2 (2M)^2 M \Delta \varepsilon/3}}\right)\\
	&= \exp\left({-\Delta 3\frac{(\varepsilon/3 - \lambda e^{-\lambda/2}(\lambda e/M)^M)^2}{8 M^3 \varepsilon}}\right).
\end{align*}\\

\noindent \textbf{Case 3}:  Finally, consider the event $|N(v) \cap N_m| > \Delta \varepsilon/3$.
\begin{align*}
\expect(|N(v) \cap N_m|)
	&=\sum_{w \in N(v)} \prob(w \in V_m)\\
	&= \sum_{w \in N(v)} \prob(\deg_X(w) < m)\\
	&\leq \sum_{w \in N(v)} e^{-(m-p\deg_G(w))^2/2p \deg_G(w)}
			&&\text{(by Lemma \ref{L:binomtail})}\\
	&\leq \sum_{w \in N(v)} e^{-(m-\lambda/2)/\lambda}\\
	&\leq \Delta e^{-(m-\lambda/2)/\lambda}.
\end{align*}
As in Case 1, changing the status of any edge changes the value of $|N(v) \cap V_N|$ by at most $2$ and the the event $|N(v) \cap V_m| \geq a$ can be certified by a collection of at most $ma$ edges.  By the choice of $\lambda$, $(m-\lambda/2)^2/\lambda \geq \ln{(3/\varepsilon)}$ and therefore, since $e^{-(m-\lambda/2)/\lambda} < \varepsilon/3$, by Theorem \ref{T:tal},
\begin{align*}
\prob(|N(v) \cap V_m| > \Delta \varepsilon/3)
	&\leq \exp \left({-\frac{(\Delta \varepsilon/3 - \Delta e^{-(m-\lambda/2)/\lambda})^2}{2 \cdot 2^2 m \Delta \varepsilon/3}}\right)\\
	&= \exp \left({-3 \Delta \frac{(\varepsilon/3 - e^{-(m-\lambda/2)/\lambda})^2}{8m}\varepsilon}\right).
\end{align*}

Set
\begin{align*}
c_0 &= \frac{3}{8 \varepsilon} \min \bigg\{\frac{(\varepsilon/3 - e^{-\lambda/2}(e\lambda/M)^M)^2}{M}, 
\frac{(\varepsilon/3 - \lambda e^{-\lambda/2}(e\lambda/M)^M)^2}{M^3},\\
	& \hspace{2.4in} \frac{(\varepsilon/3 - e^{-(m-\lambda/2)^2/\lambda})^2}{m}\bigg\}.
\end{align*}

 Then, $\prob(|N(v) \cap A| > \varepsilon \Delta) \leq 3 e^{-c_0 \Delta}$. Thus, for each $v \in V_h$, 
\[
\prob(B_v) \leq 3 e^{-c_0 \Delta}. 
\]

Let $u,v, w, z \in V_h$.  The events $A_{uv}$ and $A_{wz}$ are independent if $d(u, w) \geq 4$, $A_{uv}$ is independent of $B_w$ if $d(u, w) \geq 5$ and $B_u$ is independent of $B_w$ if $d(u,w) \geq 6$.  Therefore, when $\Delta \geq 2$, for fixed $u$ and $v$, the event $A_{uv}$ is independent of all but at most $(1+ \Delta + \Delta(\Delta-1) + \Delta(\Delta-1)^2)\Delta \leq \Delta^4$ events $A_{wz}$ and all but at most $1 + \Delta + \Delta(\Delta-1) + \Delta(\Delta-1)^2 + \Delta(\Delta-1)^3 \leq \Delta^4$ events $B_w$.  Meanwhile, the event $B_v$ is independent of all but at most $(1 + \Delta + \Delta(\Delta-1) + \Delta(\Delta-1)^2 + \Delta(\Delta-1)^3)\Delta \leq \Delta^5$ events $A_{wz}$ and all but at most $1 + \Delta + \Delta(\Delta-1) + \Delta(\Delta-1)^2 + \Delta(\Delta-1)^3 + \Delta(\Delta-1)^4 \leq \Delta^5$ events $B_w$.

Set $\gamma_1 = \ln{\Delta}/\Delta^5$ and $\gamma_2 = 1/\Delta^5$ and let $uv \in E(V_h)$.  Using the inequality $(1-t) \geq e^{-t -t^2}$, valid for $t$ small enough,
\begin{align*}
\gamma_1 (1-\gamma_1)^{\Delta^4}(1-\gamma_2)^{\Delta^4}
		& \geq \frac{\ln{\Delta}}{\Delta^5} e^{-\ln{\Delta}/\Delta(1 + \ln{\Delta}/\Delta^5)} e^{-1/\Delta(1 + 1/\Delta^5)}\\
		& \geq \frac{\ln{\Delta}}{2 \Delta^5}	
					&&\text{(for $\Delta \geq 4$)}\\
		&\geq \frac{2^{2M+d}\lambda^{m-d+1}}{\Delta^{m-d+1}}
					&&\hspace{-1.2in}\text{(for $\Delta$ large, $m-d+1 \geq 5$)}\\
		&\geq \prob(A_{u,v}).
\end{align*}
Similarly,
\begin{align*}
\gamma_2 (1-\gamma_1)^{\Delta^5}(1-\gamma_2)^{\Delta^5}
	&= \frac{1}{\Delta^5} \left(1 - \frac{\ln{\Delta}}{\Delta^5} \right)^{\Delta^5} \left(1 - \frac{1}{\Delta^5} \right)^{\Delta^5}\\
	&\geq \frac{1}{\Delta^5} e^{-\ln{\Delta}(1 + \ln{\Delta}/\Delta)}e^{-1 + 1/\Delta^5}\\
	&\geq \frac{1}{\Delta^5}\frac{1}{\Delta^2}\frac{1}{3}
							&&\text{(for $\Delta \geq 2$)}\\
	&\geq 3 e^{-c_0 \Delta}	&&\text{(for $\Delta$ large)}\\
	&\geq \prob(B_v).
\end{align*}

Therefore, since $\gamma_1, \gamma_2 \in (0,1)$, by the Local Lemma (Theorem \ref{T:LLL}), 
\[
\prob \bigg(\bigcap_{uv \in E(V_h, V_h)}\overline{A_{uv}} \cap \bigcap_{v \in V_h}\overline{B_{v}}\bigg) > 0.
\]
\end{proof}

Next, by deleting a few more edges from $G$, the vertices in $V_h$ that might not have been distinguished from one of their neighbours in $G[E \setminus E_1]$ can be made to have colour sets different from their neighbours.

\begin{lemma}\label{L:highdegB}
For each $\alpha,\beta >0$ with $\alpha> \beta$, $M > 0$ and $B \geq 2$ there is a $\Delta_3 > 0$ such that if $G = (V,E)$ is a graph with $\Delta(G)  \geq \Delta_3$, $\varphi$ is a proper total colouring of $G$ and there is a set $L \subseteq \{v \in V: \deg(v) > \alpha \Delta(G)\}$ such that if $v \in V$ with $\deg(v) > \alpha \Delta(G)$, then $|N(v) \cap L| \leq \beta \Delta(G)$, and $E_1 \subseteq E$ so that $\Delta(G[E_1]) \leq M$, then there is a set of edges, $E_2 \subseteq E \setminus E_1$ so that, setting $\varphi_2 = \varphi|_{V \cup E \setminus (E_1 \cup E_2)}$,
\begin{enumerate}[(a)]
	\item if $u \in L$ then $\deg_{E_2}(u) = B$,
	\item if $v \notin L$ and $\deg(v) > \alpha \Delta(G)$, then $|E_2 \cap E(v)| \leq B-1$, and
	\item if $u, v \in L$ are adjacent, then $C_{\varphi_2}(u) \neq C_{\varphi_2}(v)$.
\end{enumerate}
\end{lemma}

\begin{proof}
Fix $\alpha > \beta >0$, $B \geq 2$ and let $G$ be a graph with $\Delta(G) = \Delta$, with $\varphi$ a proper total colouring of $G$, $L \subseteq \{v \in V: \deg(v) > \alpha \Delta\}$ with the property that if $\deg(v) > \alpha \Delta$, then $|N(v) \cap L| \leq \beta \Delta$ and $E_1 \subseteq E$ with $\Delta(G[E_1]) \leq M$.

Note that for each $u \in L$, $|N(u) \setminus L| \geq \alpha \Delta - \beta \Delta \geq B +M$.  Select $E_2$ at random as follows:  for each $u \in L$, select a set of $B$ edges in $E \setminus E_1$ from $u$ to $N(u) \setminus L$ uniformly at random to include in $E_2$.

For each $v \notin L$ with $\deg(v) \geq \alpha \Delta$, and $u_1, u_2, \ldots u_B \in N(v) \cap L$, let $A_{v, \{u_1, u_2, \ldots, u_B\}}$ be the event that all of the edges $vu_1, vu_2, \ldots, v u_B$ belong to $E_2$.  For each $u, v \in L$ with $uv \in E$, let $B_{u,v}$ be the event that $C_{\varphi_2}(u) = C_{\varphi_2}(v)$.  Again using the Local Lemma, it is shown that the probability that none of the events $A_{v, \{u_1, u_2, \ldots, u_B\}}$ or $B_{u,v}$ occur is strictly positive and hence there is a choice of $E_2$ that satisfies the conditions $(b)$ and $(c)$.  Note that any choice of $E_2$ satisfies condition $(a)$ by construction.

Fix $v \notin L$ with $\deg(v) \geq \alpha \Delta$ and $u_1, \ldots, u_B \in N(v) \cap L$.  For each $i = 1, 2, \ldots, B$,
\begin{align*}
\prob(vu_i \in E_2)	&=\frac{\binom{\deg(u_i) - |N(u) \cap L|-M}{B-1}}{\binom{\deg(u_i) - |N(u) \cap L|-M}{B}}\\
					&=\frac{B}{\deg(u_i) - |N(u) \cap L| -M-B+1}\\
					&\leq \frac{B}{\alpha \Delta - \beta \Delta - M -B+1}\\
					&=\frac{B}{(\alpha - \beta)\Delta -M -B+1}
\end{align*}
and hence
\[
\prob(A_{v, \{u_1, u_2, \ldots, u_B\}}) \leq \left( \frac{B}{(\alpha - \beta)\Delta -B+1}\right)^B.
\]

Given $u, v \in L$ with $uv \in E$, fix $C_B \subseteq \varphi[E(u, V_{\alpha})]$ with $|C_B| = B$.  If $C_B = \varphi[E_2(u)]$, then either $\prob(C_{\varphi_2}(u) = C_{\varphi_2}(v) |\ C_B = \varphi[E_2(u)]) = 0$ or else there is exactly one set of $B$ colours $C_{v,B}$ with $C_{\varphi}(u) \setminus C_B = C_{\varphi}(v) \setminus C_{v,B}$.  Thus,
\begin{align*}
\prob(C_{\varphi_2}(u) = C_{\varphi_2}(v) |\ C_B = \varphi[E_2(u)])
	&\leq \frac{1}{\binom{\deg(v) - |N(v) \cap L| - M}{B}}\\
	&\leq \frac{1}{\binom{(\alpha - \beta)\Delta - M}{B}}\\
	&\leq \left(\frac{B}{(\alpha - \beta)\Delta - M} \right)^B
\end{align*}
and hence
\[
\prob(B_{u,v}) \leq \left(\frac{B}{(\alpha - \beta)\Delta - M} \right)^B.
\]

Two events of the form $A_{v_1, \{u_1, \ldots, u_B\}}$ and $A_{v_2, \{w_1, \ldots, w_B\}}$ are independent whenever $\{u_1, \ldots, u_B\} \cap \{w_1, \ldots, w_B\} = \emptyset$ and events $A_{v_1, \{u_1, \ldots, u_B\}}$ and $B_{u,w}$ are independent if $\{u_1, \ldots, u_B\} \cap \{u,w\} = \emptyset$.  Similarly, two events $B_{u,w}$ and $B_{u', w'}$ are independent if $\{u,w\} \cap \{u', w'\} = \emptyset$.  Thus, an event $A_{v_1, \{u_1, \ldots, u_B\}}$ is independent of all but at most $B \Delta \binom{\beta \Delta}{B-1}$ events of the type $A_{v_2, \{w_1, \ldots, w_B\}}$ and all but at most $B \beta \Delta$ events of the type $B_{u,w}$.  Similarly, an event $B_{u,w}$ is independent of all but at most $2 \Delta \binom{\beta \Delta}{B-1}$ events of the type $A_{v_2, \{w_1, \ldots, w_B\}}$ and all but $2 \beta \Delta$ events of the type $B_{u', w'}$.

Therefore, by the Local Lemma \ref{T:LLL}, since, for $\Delta$ sufficiently large
\[
\left(\frac{B}{(\alpha - \beta)\Delta -M - B+1} \right)^B \left(B \beta \Delta + B \left(\frac{\beta}{B-1} \right)^{B-1} \Delta^{B-1} + 1 \right) e \leq 1,
\]
there is a choice of $E_2$ that satisfies conditions $(a)$ and $(b)$ in the statement of the lemma.
\end{proof}

\begin{proof}[Proof of Proposition \ref{P:highdeg}]
 Set $\varepsilon = 1/3$, $m=8$, $d=4$ and let $\Delta_2 >0$ and $M>0$ be given by Lemma \ref{L:highdegA}.  Set $\alpha = 1/2$, $\beta = 1/3$, $B=2$ and let $\Delta_3$ be given by Lemma \ref{L:highdegB}.  Let $G=(V,E)$ be a graph with $\Delta(G) = \Delta \geq \max\{\Delta_2, \Delta_3\}$ and let $\varphi$ be a total $k$-colouring of $G$.

Let $E_1 \subseteq E$ be given by Lemma \ref{L:highdegA} and for $L = \{v \in V_h: \deg_{E_1}(v) < 8\}$ let $E_2 \subseteq E\setminus E_1$ be given by Lemma \ref{L:highdegB}.  By the choice of $E_1$ and $E_2$, $\Delta(G[E_1 \cup E_2]) \leq M+2$ and so by Vizing's theorem, there is a proper edge colouring, $\psi$, of $G[E_1 \cup E_2]$ with $M+3$ colours.  Let these $M+3$ colours be disjoint from the set of colours used by $\varphi$.  Define a total colouring $\varphi'$ of $G$ as follows
\[
\varphi'(x) = 
	\begin{cases}
		\varphi(x),	&\text{for $x \in V \cup E \setminus (E_1 \cup E_2)$;}\\
		\psi(x),		&\text{for $x \in E_1 \cup E_2$.}
	\end{cases}
\]

The map $\varphi'$ is a proper total $(k+M+3)$-colouring of $G$.  For each $u, v \in V_h$ with $uv \in E$, if $u \notin L$, then $|C_{\varphi_2}(u) \Delta C_{\varphi_2}(v)| \geq 4 - (2+1) >0$ and so $C_{\varphi'}(u) \neq C_{\varphi'}(v)$.  If $u, v \in L$ and $uv \in E$, then $C_{\varphi_2}(u) \neq C_{\varphi_2}(v)$ by the choice of $E_2$ and so $C_{\varphi'}(u) \neq C_{\varphi'}(v)$.
\end{proof}

\section{Proof of Theorem \ref{T:mainthm}}\label{S:mainthmpf}

\begin{proof}[Proof of Theorem \ref{T:mainthm}]
Let $G=(V,E)$ be a graph with $\Delta(G) \geq \Delta_0$ and let $\varphi$ be a proper total colouring of $G$ with $\chi''(G)$ colours.  By Proposition \ref{P:highdeg}, there is a proper total ($\chi''(G) + C_1$)-colouring of $G$ so that for each $u, v \in V_h$, if $uv \in E$, then $C_{\varphi'}(u) \neq C_{\varphi'}(v)$.

By Proposition \ref{P:lowdeg}, there is a proper total colouring $\varphi''$  with $\varphi''|_{E \cup V_h} = \varphi'|_{E\cup V_h}$ that distinguishes every vertex in $V_{\ell}$ from each of its neighbours.  By the choice of $\varphi''$, if $v \in V_h$, $C_{\varphi''}(v) = C_{\varphi'}(v)$ and hence $\varphi''$ distinguishes each vertex in $V$ from every one of its neighbours.
\end{proof}

Following through the calculations in the proofs carefully, it can be shown that for $\varepsilon = 1/3$, $m=8$, $d=4$ and $B=2$, then $\lambda$ can be taken to be $34$ and $M=81$.  While this estimate is likely not optimal, and  does not seem apply to many real-world examples, it shows that for a graph $G$ with $\Delta(G) \geq \exp(10^{58})$, then $\chi_{at}(G) \leq \chi''(G) + 84$.

\end{document}